\documentclass[12pt]{amsart}
\usepackage{amscd,amsmath,amsthm,amssymb,graphics}

\usepackage{pgf,tikz, pifont}

\usetikzlibrary{arrows}

\usepackage[all]{xy}

\def\NZQ{\Bbb}               
\def\NN{{\NZQ N}}

\def\frk{\frak}               

\def\Phi{{\frk n}}
\def\Phi{{\frk N}}
\def\opn#1#2{\def#1{\operatorname{#2}}} 
\opn\chara{char} \opn\length{\ell} \opn\pd{pd} \opn\rk{rk}
\opn\projdim{proj\,dim} \opn\injdim{inj\,dim} \opn\rank{rank}
\opn\depth{depth} \opn\grade{grade} \opn\height{height}
\opn\embdim{emb\,dim} \opn\codim{codim}

\opn\Tr{Tr} \opn\bigrank{big\,rank}
\opn\superheight{superheight}\opn\lcm{lcm}
\opn\trdeg{tr\,deg}
\opn\reg{reg} \opn\lreg{lreg} \opn\ini{in} \opn\lpd{lpd}
\opn\size{size}\opn\bigsize{bigsize}
\opn\cosize{cosize}\opn\bigcosize{bigcosize}
\opn\sdepth{sdepth}\opn\sreg{sreg}
\opn\link{link}\opn\fdepth{fdepth}
\opn\index{index}
\opn\index{index}
\opn\indeg{indeg}
\opn\N{N}
\opn\SSC{SSC}
\opn\SC{SC}
\opn\conv{conv}
\opn\div{div} \opn\Div{Div} \opn\cl{cl} \opn\Cl{Cl}
\opn\Spec{Spec} \opn\Supp{Supp} \opn\supp{supp} \opn\Sing{Sing}
\opn\Ass{Ass} \opn\Min{Min}\opn\Mon{Mon} \opn\dstab{dstab} \opn\astab{astab}
\opn\Syz{Syz}
\opn\reg{reg}
\opn\Ann{Ann} \opn\Rad{Rad} \opn\Soc{Soc}
\opn\Im{Im} \opn\Ker{Ker} \opn\Coker{Coker} \opn\Am{Am}
\opn\Hom{Hom} \opn\Tor{Tor} \opn\Ext{Ext} \opn\End{End}
\opn\Aut{Aut} \opn\id{id}

\opn\nat{nat}
\opn\pff{pf}
\opn\Pf{Pf} \opn\GL{GL} \opn\SL{SL} \opn\mod{mod} \opn\ord{ord}
\opn\Gin{Gin} \opn\Hilb{Hilb}\opn\sort{sort}
\opn\initial{init}
\opn\ende{end}
\opn\height{height}
\opn\type{type}
\opn\aff{aff} \opn\con{conv} \opn\relint{relint} \opn\st{st}
\opn\lk{lk} \opn\cn{cn} \opn\core{core} \opn\vol{vol}
\opn\link{link} \opn\star{star}\opn\lex{lex}\opn\Mon{Mon}\opn\Min{Min}
\opn\gr{gr}

\def\pot#1#2{#1[\kern-0.28ex[#2]\kern-0.28ex]}

\opn\dirlim{\underrightarrow{\lim}}
\opn\inivlim{\underleftarrow{\lim}}
\let\to=\rightarrow

\def\Implies{\ifmmode\Longrightarrow \else
        \unskip${}\Longrightarrow{}$\ignorespaces\fi}
\def\implies{\ifmmode\Rightarrow \else
        \unskip${}\Rightarrow{}$\ignorespaces\fi}
\def\iff{\ifmmode\Longleftrightarrow \else
        \unskip${}\Longleftrightarrow{}$\ignorespaces\fi}

\let\:=\colon
\newtheorem{Theorem}{Theorem}[section]
 \newtheorem{Lemma}[Theorem]{Lemma}
 \newtheorem{Corollary}[Theorem]{Corollary}
 \newtheorem{Proposition}[Theorem]{Proposition}
 \newtheorem{Remark}[Theorem]{Remark}
 
 \newtheorem{Example}[Theorem]{Example}
 
 \newtheorem{Definition}[Theorem]{Definition}
 \newtheorem*{Definition*}{Definition}

 \newtheorem*{Conjecture*}{Conjecture}

\let\epsilon\varepsilon
\let\kappa=\varkappa
\def\qed{\ifhmode\textqed\fi
      \ifmmode\ifinner\quad\qedsymbol\else\dispqed\fi\fi}
\def\textqed{\unskip\nobreak\penalty50
       \hskip2em\hbox{}\nobreak\hfil\qedsymbol
       \parfillskip=0pt \finalhyphendemerits=0}
\def\dispqed{\rlap{\qquad\qedsymbol}}

\opn\dis{dis}
\def\pnt{{\raise0.5mm\hbox{\large\bf.}}}

\opn\Lex{Lex}

\begin{document}

 \title{Further results on monomial ideals of projective dimension one}

\author{Dancheng Lu, Jiawen Shan, Yu Wang}

\address{Dancheng Lu, School of Mathematical Sciences,\allowbreak
Soochow University,\allowbreak
215006 Suzhou,\allowbreak
P.R.China}
\email{ludancheng@suda.edu.cn}

\address{Jiawen Shan, School of Mathematics and systems Science, \allowbreak
Shenyang Normal University,\allowbreak 110034
Shenyang ,\allowbreak
P.R.China}
\email{ShanJiawen@outlook.com}

\address{Yu Wang, School of Mathematical Sciences,\allowbreak
Soochow University,\allowbreak
215006 Suzhou,\allowbreak
P.R.China}
\email{20244207009@stu.suda.edu.cn}

\begin{abstract}
We prove that a monomial ideal has projective dimension one if and only if its minimal monomial generators can be ordered so that each successive colon ideal is principal, and show that this characterization is equivalent to the monomial version of the Hilbert-Burch Lemma. Furthermore, we prove that any squarefree monomial ideal of projective dimension one with a linear resolution has the property that all its powers \(I^s\) admit linear quotients, and we provide a partial classification of such ideals.
\end{abstract}
 \maketitle

 \section{Introduction}

	The monomial ideals of projective dimension 1 have been studied extensively.  It was shown in \cite{HHZ}, that a squarefree monomial ideal \(I\) has projective dimension \(\leq 1\) if and only if the simplicial complex \(\Delta:=\{F\subseteq [n] \mid \mathrm{x}_{[n]\setminus F}\in I\}\) is a quasi-tree. In \cite{FH}, it was proved in Theorem~27 that a monomial ideal \(I\) has projective dimension \(\leq 1\) if and only if \(I\) has a minimal free resolution supported on a graph tree. More recently, the minimal free resolutions of powers of squarefree monomial ideals \(I\) with projective dimension \(\leq 1\) have been studied in different directions in \cite{CFK} and \cite{CFK1} respectively.  In particular, \cite{CFK1} presents an explicit minimal free resolution for powers \(I^s\) of such ideals, supported on a polyhedral complex whose cells are cubes.

Let \(S = \mathbb{K}[x_1,\dots,x_n]\) be the polynomial ring over a field \(\mathbb{K}\). In Section 2, we prove that a monomial ideal \(I \subseteq S\) has projective dimension one if and only if its minimal generators admit a linear ordering \(u_1>u_2>\cdots>u_r\) such that the colon ideal \((u_1,\dots,u_{i-1})\colon u_i\) is principal for each \(i = 2,\dots,r\). This argument relies on the fact that any minimal free resolution of a monomial ideal of projective dimension one is supported on a tree graph, as established in \cite[Theorem~27]{FH}. In the final part of the paper, we further show that this characterization also follows from the Hilbert--Burch Lemma.

Let \(d\in \mathbb{Z}_{\geq 1}\). A monomial ideal \(I\) has a \(d\)-\emph{linear resolution} if \(\beta_{i,j}(I)=0\) for any pair \(i,j\) with \(j - i \neq d\). This is equivalent to \(I\) admitting a multigraded free resolution where the matrices of all differential maps have entries in the set of linear forms or 0, which justifies the name.
A monomial ideal \(I\) has linear quotients if there exists a linear ordering \(u_1>u_2>\cdots>u_r\) of its minimal generators such that \((u_1,\ldots,u_{i-1}):u_i\) is generated by variables for all \(i=2,\ldots,r\).
It was proved in \cite{HHBook} that if \(I\) is a graded ideal generated in a single degree with linear quotients, then \(I\) has a linear resolution. A monomial ideal $I$ is \emph{linearly related} if $\beta_{1,j}(I)=0$ for all $j>d(I)+1$, where $d(I)$ denotes the maximal degree of minimal generators of $I$.

In Section 3, we show that if \(I\) is a monomial ideal of projective dimension 1, then \(I\) has  linear quotients  whenever it has a linear resolution. Moreover, if we assume further that \(I\) is squarefree, then all its powers \(I^s\) have linear quotients.   We provide an example to illustrate that the squarefree assumption in the above result is indispensable.

To gain more insight into squarefree monomial ideals of projective dimension 1 with linear quotients, we introduce a type classification for such ideals. We explicitly describe type 1 and type 2 ideals, and show that type 2 ideals are essentially the complementary edge ideals of trees studied in \cite{HQM,FM1,LWZ}.

The Hilbert-Burch lemma, which governs the free resolutions of ideals with projective dimension one, has a long-standing research history.
In the monograph \cite{HHBook}, a version of the Hilbert-Burch lemma for monomial ideals is established. Nevertheless, its original proof is rather non-trivial and requires considerable preliminary knowledge. In Section 4, we  provide a more intuitive and purely combinatorial proof by our previous results. This proof further shows that our colon-ideal characterization of monomial ideals of projective dimension 1 is equivalent to the Hilbert-Burch lemma.

We close this section by fixing notation. Let \(S = \mathbb{K}[x_1,\dots,x_n]\) denote the polynomial ring over a field \(\mathbb{K}\). Given a monomial ideal \(I \subseteq S\), \(G(I)\) will stand for the set of its minimal monomial generators. For non-negative integers \(i,j\) and \(\alpha \in \mathbb{Z}_{\geq 0}^n\), the \(i\)-th total Betti number, the \(i\)-th graded Betti number in degree \(j\), and the \(i\)-th multigraded Betti number in multi-degree \(\alpha\) of \(I\) are denoted by \(\beta_i(I)\), \(\beta_{i,j}(I)\), and \(\beta_{i,\alpha}(I)\), respectively. We also denote the projective dimension of \(I\) by \(\operatorname{pd}(I)\).

For any monomial $u = x_1^{a_1}\cdots x_n^{a_n} \in S$, we define its total degree by $\deg(u) = a_1 + \cdots + a_n$, its degree with respect to the variable $x_i$ by $\deg_{x_i}(u) = a_i$, and its multi-degree by $\operatorname{mdeg}(u) = (a_1,\dots,a_n)$.
Furthermore, for monomials $u_1,\dots,u_r \in S$, $\operatorname{lcm}(u_1,\dots,u_r)$ and $\operatorname{gcd}(u_1,\dots,u_r)$ denote the least common multiple and greatest common divisor of $u_1,\dots,u_r$, respectively.

For a finite set $A$, $|A|$ denotes the cardinality of $A$.

\section{Colon characterization}

In this section, we apply Theorem 27 from \cite{FH} to establish equivalent characterizations for monomial ideals with projective dimension one. For this purpose, we first recall some necessary background preliminaries.

A \emph{simplicial complex} over the vertex set \([r]=\{1,2,\dots,r\}\) is a collection \(\Delta\) of subsets of \([r]\) such that (i) every singleton \(\{j\}\subseteq[r]\) belongs to \(\Delta\), and (ii) \(\Delta\) is closed under taking subsets. For any face \(F\in\Delta\), define its dimension by \(\dim F = |F|-1\). The empty face has dimension \(-1\), and \(\dim\Delta\) is the maximum dimension of all faces of \(\Delta\). Let \(\operatorname{Mon}(S)\) denote the set of all monomials of \(S\).
Given a map \(f\colon [r]\to \operatorname{Mon}(S)\) assigning a monomial to each vertex, we call \(\Delta\) an \(f\)-\emph{labeled simplicial complex}, denoted \((\Delta,f)\).
For a labeled simplicial complex \((\Delta,f)\), each face \(F\in\Delta\) is associated with a monomial \(u_F = \operatorname{lcm}\{f(k)\mid k\in F\}\), where the least common multiple is computed in the monomial lattice \(\operatorname{Mon}(S)\) of \(S\). With the above notation, we construct the multigraded free \(S\)-module complex \(\mathcal{C}_\bullet(\Delta; f)\) as follows:

For each integer \(i\geq -1\), set \[C_i = \bigoplus_{\substack{F\in\Delta\\ \dim F = i}} S e_F\] where \(e_F\) is the basis element indexed by \(F\); assign to \(e_F\) the multidegree \(\mathrm{mdeg} (e_F) = \mathrm{mdeg} (u_F)\) and set \(C_i=0\) whenever \(i>\dim\Delta\).
Endow \([r]\) with the standard linear order; for any \(F\in\Delta\) and \(k\in F\), let \(\sigma(k,F)\) denote the position of \(k\) in the increasingly ordered vertices of \(F\). For each \(i\ge 1\), define the \(S\)-linear differential \(\partial_i\colon C_i\to C_{i-1}\) by \[\partial_i(e_F) = \sum_{k\in F} (-1)^{\sigma(k,F)}\cdot \frac{u_F}{u_{F\setminus\{k\}}}\, e_{F\setminus\{k\}}\] for every \(i\)-dimensional face \(F\in\Delta\).
Identify \(C_{-1}=S\); the augmentation map \(\partial_0\colon C_0\to C_{-1}\) is given by \(\partial_0(e_{\{k\}}) = f(k)\) for all \(k\in[r]\). The resulting multigraded free \(S\)-module complex \(\mathcal{C}_\bullet(\Delta,f)\) is \[0 \to C_{d} \xrightarrow{\partial_{d}} C_{d-1} \xrightarrow{\partial_{d-1}} \cdots \xrightarrow{\partial_1} C_0 \xrightarrow{\partial_0} C_{-1} \to 0.\]

A \textit{simplicial tree} is a connected simplicial complex $\Delta$ such that every nonempty subcollection of the facets of $\Delta$ contains a \textit{leaf} of that subcollection, where a \textit{leaf} of a simplicial complex $\Gamma$ is a facet $F$ of $\Gamma$ for which there exists a facet $G\in\Gamma$, called a joint of $F$, such that $F\cap H\subseteq G$ for every facet $H\in\Gamma\setminus\{F\}$.

Every nonempty subcollection of a simplicial tree is itself a simplicial tree if and only if it is connected. The fact that simplicial trees have vanishing simplicial homologies yields the following result (see \cite[Theorem~3.2]{F}).

\begin{Proposition}\label{basic}
Let $(\Delta,f)$ be a labeled simplicial tree and $I=(f(1),\dots,f(r))$. If $G(I)=\{u_1,\dots,u_r\}$, let $L_I$ denote the lattice of elements $\mathrm{lcm}(\{u_i\mid i\in A\})$ for all $A\subseteq [r]$, ordered by divisibility. Then $\mathcal{C}_\bullet(\Delta,f)$ resolves $S/I$ if and only if for every $u\in L_I$, the induced subcomplex of $\Delta$ on $\{i\mid f(i)\mbox{ divides } u\}$ is connected.
\end{Proposition}

Graph-theoretic trees are precisely one-dimensional simplicial trees. With our notation in place, \cite[Theorem~27]{FH} may be restated as follows.

\begin{Proposition}\label{sara}
Let $I$ be a monomial ideal with $G(I)$. Then $\mathrm{pd}(I)=1$ if and only if there exist a tree $T$ and a bijection $f:V(T)\to G(I)$ such that $\mathcal{C}_\bullet(T,f)$ is a free resolution of $S/I$.
\end{Proposition}

As pointed out in \cite{FH}, if $\mathcal{C}_\bullet(T,f)$ is a free resolution of $S/I$ then it   is a minimal free resolution of $S/I$.
Given a tree $T$, a \textit{leaf vertex} (or simply a \textit{leaf}) of $T$ is a vertex of degree one in $T$. For all trees $T$, there exists an ordering $v_1,\ldots, v_r$ of vertices of $T$ such that for all $1\leq k\leq r$, the induced subgraph of $T$ on $\{v_1,\ldots, v_k\}$ is connected and $v_k$ is a leaf of this induced subgraph. Such an ordering is called a \textit{leaf ordering} of $T$.

\begin{Proposition} \label{lu}
Let \(I\subseteq S\) be a monomial ideal of projective dimension one with \(|G(I)|=r\geq 2\), and let \(T\) be a tree with a bijective map \(f:V(T)\rightarrow G(I)\) such that \(\mathcal{C}_\bullet(T,f)\) is a free resolution of \(S/I\). Assume further that \(V(T)=[r]\) and that \(1,2,\ldots,r\) is a leaf ordering of \(T\). Let \(T_k\) denote the induced subgraph of \(T\) on the vertex set \([1,k]\), set \(I_k=(f(1),\ldots,f(k))\), and let \(f_k\) be the restriction of \(f\) to \([k]\). Then for each \(1\leq k\leq r\), \(\mathcal{C}_\bullet(T_k,f_k)\) is a free resolution of \(S/I_k\).
\end{Proposition}
\begin{proof} Fix $1\leq k\leq r$. Take arbitrarily $1\leq j_1<j_2<\cdots<j_s\leq k$. Let $u=\mathrm{lcm}(f(j_1),\ldots, f(j_s))$ and let $$A_u=\{1\leq i\leq r \mid f(i) \mbox{ divides } u\}.$$ By Proposition~\ref{basic}, it suffices to show that the induced subgraph of $T_k$ on $A_u$, i.e., the induced subgraph of $T$ on $A_u\cap [1,k]$ is connected.

 To show this, fix $s,t\in A_u\cap [1,k]$ with $s\leq t$. Since $T_k$ is connected, there is a unique shortest path connecting $s,t$: $$s-s_1-\cdots-s_i-t.$$
 Since the induced subgraph of $T$ on $A_u$ is connected by Proposition~\ref{basic}, we conclude that $\{s_1,\dots,s_i\}\subseteq A_u$. Consequently, the induced subgraph of $T_k$ on $A_u$ is connected, as desired.
\end{proof}

\begin{Corollary} \label{easy} Let $I$ be a monomial ideal of projective dimension 1. Then there exists an ordering $u_1,\ldots, u_r$ of elements in $G(I)$ such that for all $2\leq k\leq r$, the following statements hold:
\begin{enumerate}
       \item  $I_k=(u_1,\ldots,u_k)$ has projective dimension 1.
       \item  There exists some $1\leq j_k<k$ such that $u_{j_k}$ divides $\operatorname{lcm}(u_i,u_k)$ for all $i<k$.
     \end{enumerate}
\end{Corollary}
\begin{proof}By Proposition~\ref{sara}, we may take \(T\) to be a tree equipped with a bijective map \(f:V(T)\rightarrow G(I)\) such that \(\mathcal{C}_\bullet(T,f)\) is a free resolution of \(S/I\). We further assume that \(V(T)=[r]\) and that \(1,2,\ldots,r\) is a leaf ordering of \(T\). Fix $2\leq k\leq r$, let \(T_k\) denote the induced subgraph of \(T\) on the vertex set \([1,k]\).

Denote $u_i=f(i)$ for $1\leq i\leq r$. Then statement (1) follows directly from Proposition~\ref{lu}.
For the proof of (2), we first notice that for any $1\leq i<k$, the induced subgraph of $T_k$ on $B_i:=\{1\leq j\leq k \mid u_j \mbox{ divides } \operatorname{lcm}(u_i,u_k)\}$ is connected by Proposition~\ref{basic}. Since $k$ is a leaf of $T_k$, the unique vertex adjacent to $k$ should belong to $B_i$. Denote this vertex by $j_k$. Then $u_{j_k}$ divides $\operatorname{lcm}(u_i,u_k)$ for all $1\leq i<k$, as required.
\end{proof}

Note that for a monomial ideal $I$, $I$ is projective (i.e., $\mathrm{pd}(I)=0$) if and only if $|G(I)|=1$.

\begin{Theorem} \label{main} Let $I\subseteq S$ be a monomial ideal with $|G(I)|\geq 2$. Then the following statements are equivalent:
\begin{enumerate}
          \item $I$ has projective dimension 1,
          \item there exists an ordering $u_1,\ldots,u_r$ of elements in $G(I)$ satisfying that for each $2\leq k\leq r$, there exists some $1\leq j_k<k$ such that $u_{j_k}$ divides $\operatorname{lcm}(u_i,u_k)$ for all $1\leq i< k$.
          \item there exists an ordering $u_1,\ldots,u_r$ of elements in $G(I)$ satisfying that for each $2\leq k\leq r$, $(u_1,\ldots,u_{k-1}):u_k$ is generated by a single monomial.
        \end{enumerate}
\end{Theorem}
\begin{proof}
$(1)\Rightarrow (2)$ It follows from Proposition~\ref{easy}.

$(2)\Rightarrow (3)$
Fix each $k\geq 2$. Since $\operatorname{lcm}(u_{j_k},u_k)$ divides $\operatorname{lcm}(u_{i},u_k)$ for all $1\leq i<k$, it follows that
\[
(u_1,\dots,u_{k-1}):u_k
=
\left(\frac{\operatorname{lcm}(u_1,u_k)}{u_k},\dots,\frac{\operatorname{lcm}(u_{k-1},u_k)}{u_k}\right)
=
\left(\frac{\operatorname{lcm}(u_{j_k},u_k)}{u_k}\right).
\]

$(3)\Rightarrow (1)$
We proceed by induction on $r$.
Set \(I_k=(u_1,\dots,u_k)\) for $1\leq k\leq r$. Since $I_{r-1}:u_r$ is generated by a single element, \(S/(I_{r-1}:u_r)\) has projective dimension at most \(1\). Thus,
\(\beta_{i}(S/(I_{r-1}:u_r))=0\) for all \(i\geq 2\).

On the other hand, we have the following short exact sequence:
\[
0\longrightarrow \frac{S}{I_{r-1}:u_r}[-\deg(u_r)]
\longrightarrow \frac{S}{I_{r-1}}
\longrightarrow \frac{S}{I_r}
\longrightarrow 0. \tag{1}
\]
By the mapping cone construction, it follows that
\[
\beta_{i}(S/I_r)\leq \beta_{i}(S/I_{r-1})+\beta_{i-1}(S/(I_{r-1}:u_r)).
\]
By the induction hypothesis, \(\beta_{i}(S/I_{r-1})=0\) for all \(i\geq 3\). This implies \(\beta_{i}(S/I_r)=0\) for all \(i\geq 3\). Consequently, \(S/I_r\) has projective dimension at most \(2\), which further yields that the projective dimension of \(I=I_r\) equals \(1\).
  \end{proof}

For a monomial ideal $I$ of projective dimension 1, its total Betti numbers are determined by the number of its minimal generators via the Hilbert-Burch Lemma. We now present a formula for the multigraded Betti numbers, and hence the graded Betti numbers of such ideals. For this purpose, we require the following definition from \cite{FTV}.
\begin{Definition}
\em Let \( I, J, K \subset S\) be monomial ideals such that \( G(I) \) is the disjoint union of \( G(J) \) and \( G(K) \). Then \( I = J + K \) is a \emph{Betti splitting} of \( I \) if
\[
\beta_{i,j}(I) = \beta_{i,j}(J) + \beta_{i,j}(K) + \beta_{i-1,j}(J \cap K) \quad \text{for all } i \in \mathbb{N} \text{ and (multi)degree } j,
\]
where \( \beta_{-1,j}(M) = 0 \) for any graded module \( M \) and all \( j \) by convention.
\end{Definition}

By \cite[Proposition~2.1]{FTV}, for monomial ideals \(I, J, K\) with \(G(I)=G(J)\sqcup G(K)\) (where \(\sqcup\) denotes the disjoint union), the decomposition \(I=J+K\) is a Betti splitting if and only if the following condition holds: for all \(i\ge 0\) and all graded degrees \(j\), the natural map
\[
\phi_{i,j}\colon \operatorname{Tor}_i^S(J\cap K,\mathbb{K})_j\to \operatorname{Tor}_i^S(J,\mathbb{K})_j\oplus \operatorname{Tor}_i^S(K,\mathbb{K})_j
\]
induced by the short exact sequence \(0\to J\cap K\to J\oplus K\to I\to 0\) is the zero map. Note that \(\phi_{0,j}\) is automatically the zero map because \(G(I)=G(J)\sqcup G(K)\).

\begin{Corollary}\label{2.8}
Let $I$ be a monomial ideal with projective dimension 1 and let $u_1, \ldots, u_r$ be an admissible order of its minimal generators. Suppose that $(u_1, \ldots, u_{k-1}) : u_k$ is generated by a monomial $w_k$ for $k = 2, \ldots, r$. Then for each $\alpha \in \mathbb{N}^n$,
$$\beta_{1,\alpha}(I) = \left| \left\{ 2 \leq k \leq r \mid \mathrm{mdeg}(w_k) + \mathrm{mdeg}(u_k) = \alpha \right\} \right|,$$
and for any $j \in \mathbb{N}$,
$$\beta_{1,j}(I) = \left| \left\{ 2 \leq k \leq r \mid \deg(w_k) + \deg(u_k) = j \right\} \right|.$$
\end{Corollary}
\begin{proof} Set $I_{k-1}=(u_1,\ldots,u_{k-1})$ for $k=2,\ldots,r$. We claim that $I_k = I_{k-1} + (u_k)$ is a Betti splitting for all $2 \leq k \leq r$. This is because $I_{k-1} \cap (u_k) = (I_{k-1} : u_k)(u_k)=(u_kw_k)$, thus a free module. This implies that the map
\[
\phi_i: \operatorname{Tor}_i(I_{k-1} \cap (u_k), \mathbb{K}) \to \operatorname{Tor}_i(I_{k-1}, \mathbb{K}) \oplus \operatorname{Tor}_i((u_k), \mathbb{K})
\]
vanishes for all $i \geq 1$, which by \cite[Proposition~2.1]{FTV} confirms the claim.

By the definition of Betti splitting, we have for all $2\leq k\leq r$,
\[
\beta_{1,\alpha}(I_k) = \beta_{1,\alpha}(I_{k-1}) + \beta_{0,\alpha}((u_k w_k)).
\]
Note that $\beta_{1,\alpha}(I_1)=0$. Summing over $k=2$ to $r$, we get
\[
\beta_{1,\alpha}(I) = \sum_{k=2}^r \beta_{0,\alpha}((u_k w_k)).
\]
Since $\beta_{0,\alpha}((u_k w_k)) = 1$ if $\alpha = \mathrm{mdeg}(u_k w_k) = \mathrm{mdeg}(w_k) + \mathrm{mdeg}(u_k)$, and $0$ otherwise, the first equality follows.
The second equality follows from the first one.
\end{proof}

This result is useful in the next section.

\section{linear properties}

In this section, we establish the equivalence of linear relations and linear quotients for monomial ideals of projective dimension one that are generated in a single degree. We then show that all powers of squarefree monomial ideals of projective dimension one with linear resolutions have linear quotients, and finally study the classification of such squarefree ideals.

\begin{Proposition} \label{linear}
Let \(I\) be a monomial ideal of projective dimension one. If \(I\) is generated in a single degree \(d\), then the following statements are equivalent:
\begin{enumerate}
  \item \(I\) has linear relations;
  \item \(I\) has a linear resolution;
  \item \(I\) has linear quotients.
\end{enumerate}
\end{Proposition}

\begin{proof}
The implications \((3)\Rightarrow (2)\Rightarrow (1)\) hold in general for graded ideals generated in a single degree.

For the proof of \((1)\Rightarrow (3)\), we adopt the notation from Corollary~\ref{2.8}. If \(I\) has linear relations, then \(\beta_{1,j}(I) = 0\) for every \(j \neq d+1\), which implies that \(\deg(w_i) = 1\) for all \(i = 2, \ldots, r\), and so \(I\) has linear quotients.
\end{proof}

In \cite{CFK1}, the minimal multigraded free resolutions of powers of a squarefree monomial ideal of projective dimension 1 were described explicitly. From this description, it is not difficult to conclude that for a squarefree monomial ideal \(I\) of projective dimension 1, if \(I\) has a linear resolution, then all its powers \(I^s\) have linear resolutions. Combining this result with Proposition~\ref{linear}, we wonder if for such an ideal, all its powers have linear quotients. This is indeed the case as shown in the next results.

Let \(\mathbb{Z}_{\geq 0}^r(s)\) denote \(\{(a_1,\ldots,a_r)\in \mathbb{Z}_{\geq 0}^r\mid a_1+\cdots+a_r=s\}\). For a monomial \(u\) and a variable \(x\), we denote by \(\deg_x(u)\) the degree of \(x\) in \(u\), which is the exponent of \(x\) in the monomial \(u\).
\begin{Lemma}\label{unique}
Let \(I\) be a squarefree monomial ideal of projective dimension 1 with \(G(I) = \{u_1, \ldots, u_r\}\). Then for any \((a_1, \ldots, a_r), (b_1, \ldots, b_r) \in \mathbb{Z}_{\geq 0}^r(s)\), if \(u_1^{a_1}\cdots u_r^{a_r} = u_1^{b_1}\cdots u_r^{b_r}\), then \(a_i = b_i\) for all \(i = 1, \ldots, r\).
\end{Lemma}
This is a copy of \cite[Proposition 4.1]{CFK1}. For completeness, we sketch a proof. By Theorem~\ref{main}, we may assume that \((u_1,\ldots,u_{r-1}):u_r\) is generated by a monomial \(w_r\). Then, for any variable \(x\) that divides  \(w_r\), \(x\) does not divide \(u_r\), but divides \(u_i\) (thus, \(\deg_x(u_i)=1\)) for all \(i=1,\ldots,r-1\). This implies \(a_1+\cdots+a_{r-1}=b_1+\cdots+b_{r-1}\), and so \(a_r=b_r\). By induction, it follows that \(a_i = b_i\) for all \(i = 1, \ldots, r-1\).
From this proof, the condition that all \(u_i\) are squarefree is essential. This uniqueness property fails in the non-squarefree case.
\begin{Example}\em
Let \(I\subset \mathbb{K}[x,y]\) with \(G(I)=\{x^2y^3, x^3y^2, x^4y\}\). Then \(I\) has projective dimension one, but \((x^2y^3)(x^4y)=(x^3y^2)^2\), which shows that the conclusion of Lemma \ref{unique} fails when \(I\) is not squarefree.
\end{Example}

\begin{Proposition}\label{power}
Let \(I\) be a squarefree monomial ideal of projective dimension one in the polynomial ring \(S = \mathbb{K}[x_1, \cdots, x_n]\) over a field \(\mathbb{K}\). If \(I\) has a \(d\)-linear resolution, then \(I^s\) has linear quotients for all \(s\ge 1\).
\end{Proposition}
\begin{proof}
According to Proposition~\ref{linear}, we may assume that \(I = (u_1, \cdots, u_r)\) with \(\deg(u_i) = d\) for all \(i = 1, \cdots, r\) such that \((u_1, \ldots, u_{i-1}) : u_i\) is generated by a single variable for \(i = 2, \ldots, r\). For \(\mathbf{a} = (a_1, \ldots, a_r)\in \mathbb{Z}_{\geq 0}^r\), we set
\[
\mathbf{u}^{\mathbf{a}} = u_1^{a_1}\cdots u_r^{a_r}.
\]
Then \(G(I^s) = \{\mathbf{u}^{\mathbf{a}} \mid \mathbf{a} \in \mathbb{Z}_{\geq 0}^r(s)\}\) by Lemma~\ref{unique}.
We order the set \(\mathbb{Z}_{\geq 0}^r(s)\) by the following rule: for any distinct vectors \((a_1, \ldots, a_r), (b_1, \ldots, b_r) \in \mathbb{Z}_{\geq 0}^r(s)\), we say \((a_1, \ldots, a_r) \leq (b_1, \ldots, b_r)\) if and only if there exists some integer \(i\) with \(1 \leq i \leq r\) such that \(a_i > b_i\) and \(a_j = b_j\) for all \(j = i+1, \ldots, r\) (this is a reverse lexicographic order). Then, accordingly, the minimal generating set \(G(I^s)\) can be ordered as \(\mathbf{u}^{\mathbf{a}_1}, \mathbf{u}^{\mathbf{a}_2}, \ldots, \mathbf{u}^{\mathbf{a}_m}\), where \(\mathbf{a}_1 \geq \mathbf{a}_2 \geq \cdots \geq \mathbf{a}_m\) under the above order, with \(m\) being the number of elements in \(\mathbb{Z}_{\geq 0}^r(s)\). We proceed by induction on \(r\) to show that \(I^s\) has linear quotients with respect to this order.

For the base case when \(r=1\), there is nothing to prove. Suppose next that \(r\geq 2\). Choose an arbitrary generator \(\mathbf{u}^{\mathbf{b}} \in G(I^s)\) with exponent vector \(\mathbf{b}=(b_1,b_2,\dots,b_r)\). Let \(I_{\mathbf{b}}\) be the ideal generated by all monomials \(\mathbf{u}^{\mathbf{a}} \in G(I^s)\) satisfying \(\mathbf{a} > \mathbf{b}\). If \(b_r = 0\), then every minimal monomial generator of \(I_{\mathbf{b}}\) is of the form \(u_1^{a_1}\cdots u_{r-1}^{a_{r-1}}\) for some \((a_1,\dots,a_{r-1})\in \mathbb{Z}_{\geq 0}^{r-1}(s)\) with $(a_1,\ldots,a_{r-1})>(b_1,\ldots,b_{r-1})$. By the induction hypothesis, the colon ideal \(I_{\mathbf{b}}: \mathbf{u}^{\mathbf{b}}\) is therefore generated by variables.

Suppose now that \(b_r > 0\). We write \(I_{\mathbf{b}} = J_{\mathbf{b}} + K_{\mathbf{b}}\), where \(J_{\mathbf{b}}\) is generated by the monomials \(\mathbf{u}^{\mathbf{a}} \in G(I^s)\) with \(\mathbf{a} > \mathbf{b}\) and \(a_r = b_r\), and \(K_{\mathbf{b}}\) is generated by the monomials \(\mathbf{u}^{\mathbf{a}} \in G(I^s)\) with \(\mathbf{a} > \mathbf{b}\) and \(a_r < b_r\). Set \(\mathbf{b'}= (b_1, \ldots, b_{r-1}, 0)\in \mathbb{Z}_{\geq 0}^{r}(s-b_r)\). Then \(J_{\mathbf{b}} : \mathbf{u}^{\mathbf{b}} = I_{\mathbf{b}'} : \mathbf{u}^{\mathbf{b}'}\), which is again generated by variables by the induction hypothesis. Here, \(I_{\mathbf{b}'}\) is the monomial ideal generated by \(\mathbf{u}^{\mathbf{a}}\) with \(\mathbf{a}\in \mathbb{Z}_{\geq 0}^{r}(s-b_r)\) and \(\mathbf{a}>\mathbf{b}'\).

Assume \((u_1, \ldots, u_{r-1}) : u_r\) is generated by a variable \(x\). We claim that \(K_{\mathbf{b}} : \mathbf{u}^{\mathbf{b}}\) is also generated by the variable \(x\). Note that there exists some \(1 \leq j_r < r\) such that \((u_{j_r}) : u_r = (x)\). Let \(\mathbf{a} = \mathbf{b} - \mathbf{e}_r + \mathbf{e}_{j_r}\), where \(\mathbf{e}_i\) denotes the standard basis vector with 1 in the \(i\)-th component and 0 elsewhere. Then \(\mathbf{a} > \mathbf{b}\) and \[(\mathbf{u}^{\mathbf{a}}) : \mathbf{u}^{\mathbf{b}}=(u_{j_r}) : u_r = (x),\]  which implies \(x \in K_{\mathbf{b}} : \mathbf{u}^{\mathbf{b}}\).

For any \(\mathbf{a} \in \mathbb{Z}_{\geq 0}^r(s)\) satisfying \(\mathbf{a} > \mathbf{b}\) and \(a_r < b_r\), we claim that \(x\) divides \(\frac{\operatorname{lcm}(\mathbf{u}^{\mathbf{a}}, \mathbf{u}^{\mathbf{b}})}{\mathbf{u}^{\mathbf{b}}}\). Indeed, since \(\deg_x(u_i) = 1\) for all \(1 \leq i \leq r-1\) and \(\deg_x(u_r) = 0\), we compute
\[
\deg_x(\mathbf{u}^{\mathbf{a}}) = a_1 + \cdots + a_{r-1} = s - a_r,\qquad \deg_x(\mathbf{u}^{\mathbf{b}}) = b_1 + \cdots + b_{r-1} = s - b_r.
\]
The condition \(a_r < b_r\) yields \(s - a_r > s - b_r\), which implies
\[
\deg_x\left( \frac{\operatorname{lcm}(\mathbf{u}^{\mathbf{a}}, \mathbf{u}^{\mathbf{b}})}{\mathbf{u}^{\mathbf{b}}} \right) > 0,
\]
verifying the claim. Consequently, \(\frac{\operatorname{lcm}(\mathbf{u}^{\mathbf{a}}, \mathbf{u}^{\mathbf{b}})}{\mathbf{u}^{\mathbf{b}}}\) lies in the principal ideal \((x)\).

Hence, the colon ideal \(K_{\mathbf{b}} : \mathbf{u}^{\mathbf{b}}\) is generated by \(x\). As a result, \(I_{\mathbf{b}} : \mathbf{u}^{\mathbf{b}} = J_{\mathbf{b}} : \mathbf{u}^{\mathbf{b}} + K_{\mathbf{b}} : \mathbf{u}^{\mathbf{b}}\) is generated by variables, which completes the argument.
\end{proof}

The assumption that \(I\) is squarefree cannot be removed in Proposition~\ref{power}.

\begin{Example}\em
Let \(I\subset S=\mathbb{K}[x_1,\ldots,x_5]\) be the monomial ideal with
\[
G(I)=\{x_1^2x_2^2,x_1x_2^2x_3, x_2^2x_3x_4,x_2x_3x_4x_5,x_3^2x_4x_5\}.
\]
Then \(I\) has linear quotients with respect to the given order, and under this ordering each successive colon ideal is generated by a variable. Thus, \(I\) has projective dimension 1. However, \(I^2\) does not admit a linear resolution. In fact, \(I^2\) has the following minimal free resolution by Macaulay2:
\[
0\rightarrow S(-12)\rightarrow S(-11)^2\oplus S(-10)^6\rightarrow S(-10)\oplus S(-9)^{20}\rightarrow S(-8)^{15}\rightarrow I^2\rightarrow0.
\]
\end{Example}
By the discussion above, squarefree monomial ideals of projective dimension 1 that have a linear resolution possess nice properties. We call such ideals \emph{nice ideals}.
In the remainder of this section, we devote ourselves to classifying nice ideals. First, we recall the definitions of support for a monomial and a monomial ideal: for any monomial \(u\), \(\operatorname{supp}(u)\) denotes the set of variables dividing \(u\); for a monomial ideal \(I\), \(\operatorname{supp}(I)\) is the union of \(\operatorname{supp}(u)\) for all monomials \(u \in G(I)\), where \(G(I)\) denotes the minimal generating set of \(I\).
\begin{Definition} \label{type}
Let \(I\) be a nice ideal, with minimal generating set \(G(I) = \{u_1, \ldots, u_r\}\) (where \(r \geq 2\)). Let \(I_i = (u_1, \ldots, u_i)\) for \(i = 1, \ldots, r\). Assume further that for each \(i = 2, \ldots, r\), the colon ideal \(I_{i-1} : u_i\) is generated by a single variable, and \(\deg(u_i) = d\) for all \(i = 1, \ldots, r\). We say that \(I\) is \emph{of type} \(t\) if \(|\operatorname{supp}(I)| = d + t\), denoted by \(\mathrm{type}(I)=t\).
\end{Definition}
We observe that \(|\operatorname{supp}(I_2)| = d + 1\) and \(|\operatorname{supp}(I_{i+1})| \leq |\operatorname{supp}(I_i)| + 1\) for all \(i = 1, \ldots, r-1\). It follows that \(d + 1 \leq |\operatorname{supp}(I)| = |\operatorname{supp}(I_r)| \leq d + r - 1\), so    $$1\leq \mathrm{type}(I) \leq r-1.$$ The structure of the class of nice ideals with type 1 or 2 can be explicitly described.
\begin{Proposition}
Let \(I\) be a nice ideal. Let \(\alpha\) denote the product of all variables in \(\operatorname{supp}(I)\).
Then \(\mathrm{type}(I)=1\) if and only if there exists a subset \(A\subseteq \operatorname{supp}(I)\) with \(|A|=r\) such that \(I\) is generated by the monomials \(\dfrac{\alpha}{x}\) for all \(x\in A\).
\end{Proposition}
\begin{proof}
The necessity is clear. Let \(I_A\) be the monomial ideal generated by \(\frac{\alpha}{x}\) for all \(x\in A\). For any distinct \(x,y\in A\), we have
\[
\left(\frac{\alpha}{x}\right): \frac{\alpha}{y} = (y).
\]
Thus, \(I_A\) is a nice ideal (i.e., it has projective dimension \(1\) and admits linear quotients with respect to any ordering of its minimal generators). This completes the proof of sufficiency.
\end{proof}
Let \( G \) be a simple graph with vertex set \( [n] = \{1, 2, \ldots, n\} \) and edge set \( E(G) \). Recall that the \emph{complementary edge ideal} of \( G \) is defined as
\[
I_c(G) = \left( \frac{\mathbf{x}_{[n]}}{x_jx_k} : \{j, k\} \in E(G) \right),
\]
where \( \mathbf{x}_{[n]} = \prod_{j=1}^n x_j \). Ficarra and Moradi \cite[Corollary 3.2]{FM1} and Hibi et al. \cite[Theorem 2.2]{HQM} showed that \( I_c(G) \) has a linear resolution if and only if \( G \) is connected. For connected graphs \(G\), it was proved in \cite[Proposition 4.9]{LWZ} that \( I_c(G) \) has projective dimension 1 if and only if \( G \) is a tree. Thus, if \(G\) is a tree, then \(I_c(G)\) is a nice ideal of type 2. The converse of this result holds essentially.
\begin{Proposition}
Let \(I\) be a nice ideal. Then \(\mathrm{type}(I)=2\) if and only if there exists a monomial \(u\) and a tree \(G\) such that \(I = uI_c(G)\). Moreover, \(I=I_c(G)\) (i.e., \(u=1\)) if and only if \(r = d+1\).
\end{Proposition}
\begin{proof}
The sufficiecy follows from the discussion preceding the proposition. For the converse, we assume \(I\) is a nice ideal of type 2. Set \(u = \gcd(u_1, \ldots, u_r)\) and write \(u_i = u w_i\) for all \(i = 1, \ldots, r\). Let \(J = (w_1, \ldots, w_r)\) and denote \(k = \deg(u)\). Then \(|\operatorname{supp}(J)| = d+2 - k\). Let \(\operatorname{supp}(J) = \{x_1, \ldots, x_{d+2 - k}\}\) and define a graph \(G\) as follows:

\(V(G) = \operatorname{supp}(J)\) and \(E(G) = \{V(G) \setminus \operatorname{supp}(w_i) \mid i = 1, \ldots, r\}\).

Since \(\gcd(w_1, \ldots, w_r) = 1\), for any \(j \in [d+2 - k]\), there exists some \(i \in [r]\) such that \(x_j \notin \operatorname{supp}(w_i)\), which implies \(x_j \in V(G) \setminus \operatorname{supp}(w_i)\). Thus, \(x_j\) belongs to some edge of \(G\), so \(G\) contains no isolated vertices.
Note that \(J\) is the complementary edge ideal of \(G\). Since \(J\) has linear quotients, \(G\) is connected. Furthermore, since \(J\) has projective dimension 1, \(G\) is a tree by \cite[Proposition 4.9]{LWZ}. This proves the first conclusion.

For the second conclusion, we first note that for a tree \(G\), the number of edges satisfies \(|E(G)| = |V(G)| - 1\). Since \(|E(G)| = r\) and \(|V(G)| = d+2 - k\), we have \(r = (d+2 - k) - 1 = d+1 - k\). Thus, \(I=I_c(G)\) (i.e., \(u=1\)) if and only if \(k = 0\), which is equivalent to \(r = d+1\), as required.
\end{proof}

\section{Reproof of the Hilbert-Burch Lemma}

 In this section, using Theorem~\ref{main}, we present a purely combinatorial proof of the celebrated Hilbert-Burch lemma in the case of monomial ideals.  we see that We begin by fixing the necessary notation and restating the monomial Hilbert-Burch lemma as presented in \cite{HHBook}.

Let $I \subset S$ be a monomial ideal with $G(I) = \{u_1, \dots, u_r\}$, where $r \geq 2$. We introduce the $\binom{r}{2} \times r$ matrix
	\[
	A_I = (a_k^{(i,j)})_{1 \leq i < j \leq r, \, 1 \leq k \leq r}
	\]
	whose entries $a_k^{(i,j)} \in S$ are
	\[
	a_i^{(i,j)} =	\frac{u_j}{\mathrm{gcd}(u_i,u_j)}, \quad a_j^{(i,j)} = -\frac{u_i}{\mathrm{gcd}(u_i,u_j)}, \quad \text{and} \quad a_k^{(i,j)} = 0 \ \text{if} \ k \notin \{i,j\}
	\]
	for all $1 \leq i < j \leq r$ and for all $1 \leq k \leq r$.

We now present a slightly modified variant of \cite[Lemma~9.2.4]{HHBook}, as follows.
\begin{Theorem}
\label{Hilbert}
Let \(I \subset S\) be a monomial ideal with \(G(I) = \{u_1, \dots, u_r\}\) and \(r \geq 2\).
The following conditions are equivalent:
\begin{enumerate}
    \item[(a)] \(\pd(I) = 1\);
    \item[(b)] there exists an \((r-1)\times r\) submatrix \(A_I^\#\) of \(A_I\) such that
    \[
    |\det(A_I^\#(j))| = \frac{u_j}{g_r} \quad \text{for all } j=1,2,\ldots,r.
    \]
    Here, \(A_I^\#(j)\) denotes the square matrix obtained by removing the \(j\)-th column of \(A_I^\#\), \(|\cdot|\) denotes the monomial part obtained by discarding unit coefficients, and \(g_r = \gcd(u_1,\dots,u_r)\).
\end{enumerate}
\end{Theorem}

\begin{Remark}\em
The original statement of (b) requires that, after relabeling the rows of $A_I^\#$ if necessary,
\[
(-1)^j \det(A_I^\#(j)) = \frac{u_j}{g_r} \quad \text{for all } j.
\]
We adopt the present, simpler formulation.
\end{Remark}

To prove this result, especially the implication \((b) \Rightarrow (a)\), we need some combinatorial preparations.
\begin{Lemma}\label{key} Let \(a_i\in \mathbb{Z}_{\geq 0}\) for \(i=1,\ldots,k\) and \(1\leq j_k<k\). Then

\noindent$\mathrm (1)$ If \(a_{j_k}\leq \max\{a_i,a_k\}\) for all \(i=1,\ldots,k-1\), we have $$\min\{a_1,\ldots,a_k\}+a_{j_k}=\min\{a_1,\ldots,a_{k-1}\}+\min\{a_{j_k},a_k\}.$$
$\mathrm (2)$ If \(a_{j_k}> \max\{a_i,a_k\}\) for some \(i\in \{1,\ldots,k-1\}\), we have $$\min\{a_1,\ldots,a_k\}+a_{j_k}>\min\{a_1,\ldots,a_{k-1}\}+\min\{a_{j_k},a_k\}.$$
$\mathrm (3)$ In general, $$\min\{a_1,\ldots,a_k\}+a_{j_k}\geq \min\{a_1,\ldots,a_{k-1}\}+\min\{a_{j_k},a_k\}.$$
\end{Lemma}
\begin{proof} We prove (1) by distinguishing two cases. When \(a_{j_k}\leq a_k\), the left-hand side becomes \(\min\{a_1,\ldots,a_{k-1}\}+a_{j_k}\). Since $\min\{a_{j_k},a_k\}=a_{j_k}$, the equality holds. When \(a_{j_k}>a_k\), then \(\min\{a_{j_k},a_k\}=a_{k}\). By assumption, \(a_{j_k}\leq\max\{a_i,a_k\}\) for all \(i<k\), and since \(a_{j_k}>a_k\), this implies \(a_{j_k}\leq a_i\) for all \(i<k\). Thus \(\min\{a_1,\ldots,a_{k-1}\}=a_{j_k}\), and the right-hand side becomes \(a_{j_k}+a_k\). On the other hand, \(\min\{a_1,\ldots,a_k\}=a_k\), so the equality holds.

For (2), let \(a_{k_0}=\min\{a_1,\ldots,a_{k-1}\}\), where $1\leq k_0\leq k-1$. By assumption, we have \(a_{j_k}>a_{k_0}\) and \(a_{j_k}>a_k\). It follows that $$\min\{a_1,\ldots,a_k\}+a_{j_k}=\min\{a_k+a_{j_k}, a_{k_0}+a_{j_k}\}>a_{k_0}+a_k,$$ as required.

Conclusion (3) follows directly from (1) and (2).
\end{proof}
\begin{Lemma} \label{key2} Let \(r\geq 3\), \(a_k\in \mathbb{Z}_{\geq 0}\) for \(k=1,\ldots,r\), and \(1\leq j_k<k\) for \(k=2,\ldots,r\). The following are equivalent:

$\mathrm (1)$ \(\min\{a_1,\ldots,a_r\}+a_{j_r}+\cdots+a_{j_3}=\min\{a_{j_r},a_r\}+\cdots+\min\{a_{j_3},a_3\}+\min\{a_1,a_2\}\);

$\mathrm (2)$ \(a_{j_k}\leq \max\{a_i, a_k\}\) for all \(i=1,\ldots,k-1\) and \(k=3,\ldots,r\).
\end{Lemma}
\begin{proof} \((2)\Rightarrow (1)\) We use induction on \(r\).
For \(r=3\), Lemma~\ref{key} (with \(k=3\)) gives $$\min\{a_1,a_2,a_3\}+a_{j_3}=\min\{a_{j_3},a_3\}+\min\{a_1,a_2\},$$ which proves the base case.
Suppose \(r>3\). By induction hypothesis, $$\min\{a_1,\ldots,a_{r-1}\}+a_{j_{r-1}}+\cdots+a_{j_3}=\min\{a_{j_{r-1}},a_{r-1}\}+\cdots+\min\{a_{j_3},a_3\}+\min\{a_1,a_2\}.$$
By Lemma~\ref{key}, $$\min\{a_1,\ldots,a_r\}=\min\{a_1,\ldots,a_{r-1}\}-a_{j_r}+\min\{a_{j_r},a_r\}.$$
Summing these two equalities yields the desired result.

\((1)\Rightarrow (2)\) Assume for contradiction that there exists \(3 \leq k \leq r\) such that \(a_{j_k} > \max\{a_i, a_k\}\) for some \(i = 1, \ldots, k-1\). Let \(k_0\) be the minimal such \(k\).
By the induction step of \((2)\Rightarrow (1)\), $$\min\{a_1,\ldots,a_{k_0-1}\}+a_{j_{k_0-1}}+\cdots+a_{j_3}=\min\{a_{j_{k_0-1}},a_{k_0-1}\}+\cdots+\min\{a_{j_3},a_3\}+\min\{a_1,a_2\}$$ (if \(k_0=3\), this reduces to \(0=0\)).
By the choice of \(k_0\) and Lemma~\ref{key} (2), $$\min\{a_1,\ldots,a_{k_0}\}+a_{j_{k_0}}>\min\{a_1,\ldots,a_{k_0-1}\}+\min\{a_{j_{k_0}},a_{k_0}\}.$$
Summing this inequality with the previous equality gives $$\min\{a_1,\ldots,a_{k_0}\}+a_{j_{k_0}}+\cdots+a_{j_3}>\min\{a_{j_{k_0}},a_{k_0}\}+\cdots+\min\{a_{j_3},a_3\}+\min\{a_1,a_2\}.$$
For \(k=k_0+1,\ldots,r\), Lemma~\ref{key} (3) implies $$\min\{a_1,\ldots,a_{k}\}+a_{j_{k}}\geq\min\{a_1,\ldots,a_{k-1}\}+\min\{a_{j_{k}},a_{k}\}.$$
Summing these inequalities and canceling common terms, we obtain $$\min\{a_1,\ldots,a_r\}+a_{j_r}+\cdots+a_{j_3}>\min\{a_{j_r},a_r\}+\cdots+\min\{a_{j_3},a_3\}+\min\{a_1,a_2\},$$ contradicting (1). This completes the proof.
\end{proof}
	
	   Lemmas~\ref{key} and \ref{key2} can be translated into the statements regarding monomials.

\begin{Lemma}\label{second}
Let $u_1,\ldots,u_r$ be monomials with $r\geq 2$. For each integer $k$ satisfying $2\leq k\leq r$, let $j_k$ be an integer such that $1\leq j_k<k$. Suppose the equality
\[
\gcd(u_1,\ldots,u_r)\cdot u_{j_r}\cdot u_{j_{r-1}}\cdots u_{j_3}
=\gcd(u_{j_r},u_r)\cdot \gcd(u_{j_{r-1}},u_{r-1})\cdots \gcd(u_1,u_2)
\]
holds. Then $u_{j_k}\mid \operatorname{lcm}(u_i,u_k)$ for all $2\leq k\leq r$ and all $1\leq i<k$.

\end{Lemma}

\begin{proof}
Since the equality of monomials is equivalent to the equality of their degrees in each variable, it suffices to prove the statement for the exponents of an arbitrary fixed variable $x$.
Let \(a_k = \deg_x(u_k)\) for \(k=1,\ldots,r\). Then the given equality becomes
\[
\min\{a_1,\ldots,a_r\} + a_{j_r} + \cdots + a_{j_3}
= \min\{a_{j_r},a_r\} + \cdots + \min\{a_1,a_2\},
\]
and the divisibility \(u_{j_k} \mid \operatorname{lcm}(u_i,u_k)\) is equivalent to
\[
a_{j_k} \leq \max\{a_i,a_k\} \quad \text{for all } 1\leq i<k.
\]
Since the assertion holds trivially for \(r=2\), we may assume \(r\geq 3\).
Now the conclusion follows from Lemma~\ref{key2}.
\end{proof}

\begin{Lemma}\label{simple}
Let \(u_1,\ldots,u_k\) be monomials with \(k\geq 2\). If there exists \(1\leq j_k<k\) such that \(u_{j_k}\mid \operatorname{lcm}(u_i,u_k)\) for all \(1\leq i<k\), then
\[
\frac{\gcd(u_1,\ldots,u_{k-1})}{\gcd(u_1,\ldots,u_k)}=\frac{\operatorname{lcm}(u_k,u_{j_k})}{u_k}=\frac{u_{j_k}}{\gcd(u_k,u_{j_k})}.
\]
\end{Lemma}

\begin{proof}
Fix an arbitrary variable \(x\) and let \(a_i = \deg_x(u_i)\) for \(i = 1, \ldots, k\). The condition \(u_{j_k}\mid \operatorname{lcm}(u_i,u_k)\) is equivalent to \(a_{j_k}\leq \max\{a_i,a_k\}\) for all \(i=1,\ldots,k-1\). By Lemma~\ref{key} (1), we have
\[
\min\{a_1,\ldots,a_k\} = \min\{a_1,\ldots,a_{k-1}\} + a_k - \max\{a_{j_k},a_k\}.
\]
This equality of exponents corresponds exactly to the monomial identities in the statement. Since this holds for every variable \(x\), the conclusion follows.
\end{proof}

\begin{proof}[Proof of Theorem~\ref{Hilbert}] Since the case $r=2$ is clear, we assume that $r\geq 3$ from the beginning.

$(b)\Rightarrow (a)$ Assume that  there exists an $(r-1)\times r$ submatrix $A_I^\#$ of $A_I$ such that
    $|\det(A_I^\#(j))| = \frac{u_j}{g_r}$  for all $j=1,\ldots,r$. Let \(\Gamma\) be the graph with vertex set
\[
V(\Gamma)=\{1,\ldots,r\}
\]
and edge set
\[
E(\Gamma)=\bigl\{\{i,j\}\mid \text{the row indexed by }\{i,j\} \text{ is a row of }A_I^\#\bigr\}.
\]
Then \(\Gamma\) has \(r\) vertices and \(r-1\) edges, since \(A_I^\#\) has exactly \(r-1\) rows. Clearly, \(\Gamma\) has no isolated vertices, because if $k$ is an isolated vertex, then the column of $A_I^\#$ indexed by $k$ is a zero vector, which is contradicted to the assumption that
\(A_I^\#(j)\) has nonzero determinant for every \(1\le j\le r\).

We claim that \(\Gamma\) is connected. Suppose, to the contrary, that \(\Gamma\) is disconnected. Since \(\Gamma\) has \(r\) vertices and \(r-1\) edges, it must have a connected component \(C\) satisfying \(|V(C)|>|E(C)|\). Without loss of generality, let \(V(C)=\{1,\ldots,k\}\), so \(2\le k<r\). Consider the submatrix \(A\) of \(A_I^\#\) consisting of the columns indexed by \(1,\ldots,k\). The number of nonzero rows of \(A\) is exactly \(|E(C)|\). As \(|E(C)|<k\), the \(k\) columns of \(A\) are linearly dependent. Here, all entries of $A_I^\#$ are regarded as the elements of the fractional field of $S$. Consequently, \(A_I^\#(r)\) has zero determinant, because it contains these \(k\) linearly dependent columns. This contradiction shows that \(\Gamma\) is connected.

A connected graph with \(r\) vertices and \(r-1\) edges is a tree, so \(\Gamma\) is a tree on the vertex set \([r]\). Let \(i_1,\dots,i_r\) be a leaf ordering of \(\Gamma\). That is, for each \(j=2,\dots,r\), the induced subgraph of \(\Gamma\) on \(\{i_1,\dots,i_j\}\) is connected, and \(i_j\) is a leaf vertex of this subgraph. Thus, for every \(k=2,\dots,r\), there exists a unique vertex \(j_k\in \{i_1,\dots,i_{k-1}\}\) such that \(\{j_k,i_k\}\) forms an edge of the induced subgraph on \(\{i_1,\dots,i_k\}\). In particular, \(j_2=i_1\).

 For \(k = r, r-1, \dots, 2\), let \(B_k\) denote the submatrix of \(A_I^\#\) obtained by taking the entries lying at the intersections of rows indexed by \(\{j_2,i_2\},\dots,\{j_k,i_k\}\) and columns indexed by \(i_1,\dots,i_k\), with their original order preserved.  Note that the column indexed by $i_k$ of the matrix $B_k$ contains only a nonzero element, whose monomial part is $\frac{u_{j_k}}{\gcd(u_{j_k}, u_{i_k})}$. Expanding the determinant along the  column indexed by $i_r$ of $B_r(i_1)=A_I^\#(i_1)$, we obtain
\[
|\det(A_I^\#(i_1))|=\biggl(\frac{u_{j_r}}{\gcd(u_{j_r},u_{i_r})}\biggr)\bigl|\det\bigl(B_{r-1}(i_1)\bigr)\bigr|.
\]
Expanding the determinant along the column indexed by $i_{r-1}$ of $B_{r-1}(i_1)$, we obtain
\[
|\det(B_{r-1}(i_1))|=\biggl(\frac{u_{j_{r-1}}}{\gcd(u_{j_{r-1}},u_{i_{r-1}})}\biggr)\bigl|\det\bigl(B_{r-2}(i_1)\bigr)\bigr|.
\]
Continuing in this way, we finally get
\[
|\det(A_I^\#(i_1))|=\frac{u_{j_r}}{\gcd(u_{j_r},u_{i_r})}\cdot \frac{u_{j_{r-1}}}{\gcd(u_{j_{r-1}},u_{i_{r-1}})}\cdots \cdot\frac{u_{j_{3}}}{\gcd(u_{j_{3}},u_{i_3})}\cdot\frac{u_{i_1}}{\gcd(u_{i_1},u_{i_2})}.
\]
Since $|\det(A_I^\#(i_1))|=\frac{u_{i_1}}{g_r}$, it follows that
\[
g_r \cdot u_{j_r} \cdot u_{j_{r-1}} \cdots u_{j_3} = \gcd(u_{j_r},u_{i_r}) \cdot \gcd(u_{j_{r-1}},u_{i_{r-1}}) \cdots \gcd(u_{j_3},u_{i_3}) \cdot \gcd(u_{i_1},u_{i_2}).
\]
By Lemma~\ref{second}, we have that $u_{j_k}$ divides $\operatorname{lcm}(u_i, u_{i_k})$ for all $i=i_1,\ldots,i_{k-1}$ and each $k=2,\ldots,r$. From this, we conclude that $\operatorname{pd}(I)=1$ by Theorem~\ref{main}.

$(a)\Rightarrow (b)$  By Theorem~\ref{main}, there exists a linear ordering \(u_{i_1}>\ldots>u_{i_r}\) of \(G(I)\) such that for each \(2\leq k\leq r\),
there exists \( j_k\in \{i_1,\ldots,i_{k-1}\}\) with \(u_{j_k}\mid \operatorname{lcm}(u_{i},u_{i_k})\) for all \(i=i_1,\ldots, i_{k-1}\).
Let \(g_k=\gcd(u_{i_1},\ldots,u_{i_k})\) for \(k=2,\ldots,r\). For each integer \(2\le k\le r\), let \(B_k\) be the submatrix of \(A_I\) formed by the entries lying at the intersections of the rows indexed by \(\{i_1,i_2\}=\{j_2,i_2\},\,\{j_3,i_3\},\,\dots,\,\{j_k,i_k\}\) and the columns indexed by \(i_1,i_2,\ldots,i_k\).  Thus, \(B_k\) is a \((k-1)\times k\) matrix. Let \(B_k(i_j)\) denote the submatrix of \(B_k\) obtained by removing its  column indexed by $i_j$ for \(j=1,\ldots,k\). We claim that
\[
|\det\bigl(B_k(i_j)\bigr)|=\frac{u_{i_j}}{g_k}
\]
for all \(k=2,\ldots,r\) and \(j=1,\ldots,k\).

We prove this claim by induction on \(k\). For the base case \(k=2\), we have
\[
B_2=\left( \frac{\operatorname{lcm}(u_{i_1},u_{i_2})}{u_{i_1}},\; -\frac{\operatorname{lcm}(u_{i_1},u_{i_2})}{u_{i_2}} \right).
\]
Hence,
\[
|\det B_2(i_1)|=\frac{u_{i_1}}{g_2},\quad
|\det B_2(i_2)|=\frac{u_{i_2}}{g_2},
\]
which verifies the claim.

Now suppose \(3\leq k\leq r\). Then \(B_{k-1}\) is a submatrix of \(B_k\) obtained by deleting the row indexed by $\{j_k,i_k\}$ and the column indexed by $i_k$.
By the induction hypothesis,
\[
|\det\bigl(B_{k-1}(i_j)\bigr)|=\frac{u_{i_j}}{g_{k-1}}
\]
for \(j=1,\ldots,k-1\). Note that the column indexed by $i_k$ of $B_k(i_j)$ has only one nonzero entry, it follows from standard linear algebra that for each \(j=1,\dots,k-1\),
\[
\begin{split}
|\det\bigl(B_k(i_j)\bigr)|
&=\biggl(\frac{u_{j_k}}{\gcd(u_{j_k},u_{i_k})}\biggr)\det\bigl(B_{k-1}(j)\bigr) \\
&= \frac{u_{j_k}}{\gcd(u_{j_k},u_{i_k})}\cdot\frac{u_{i_j}}{g_{k-1}} \\
&= \frac{u_{i_j}}{g_k}.
\end{split}
\]
The last equality follows from Lemma~\ref{simple}. Note that the row indexed by $\{j_k,i_k\}$ of $B_k(i_k)$ has only one nonzero entry, expanding $B_k(i_k)$ along this row, we get
\[
|\det\bigl(B_k(i_k)\bigr)|
=\left(\frac{u_{i_k}}{\gcd(u_{j_k},u_{i_k})}\right)\frac{u_{j_k}}{g_{k-1}}
=\frac{u_{i_k}}{g_k}.
\]
The last equality also relies on Lemma~\ref{simple}.
The claim is therefore established.

Finally, set \(A_I^{\#}=B_r\). Then, we conclude by the claim:
\[
|\det\bigl(A_I^{\#}(i_j)\bigr)|=\frac{u_{i_j}}{g_r}
\]
for all \(j=1,\ldots,r\). Since $\{i_1,\ldots,i_r\}=\{1,\ldots,r\}$,  we obtain $|\det\bigl(A_I^{\#}(j)\bigr)|=\frac{u_{j}}{g_r}$ for all \(j=1,\ldots,r\), completing the proof.
\end{proof}

We remark that the proof of \((b)\Rightarrow (a)\) provides an alternative argument for the colon-ideal characterization of monomial ideals with projective dimension one. This demonstrates that such a colon-ideal characterization is essentially equivalent to the Hilbert-Burch Lemma.

\end{document}